\documentclass[a4paper,11pt]{article}

\usepackage{verbatim}
\usepackage{amsthm, amssymb, amsmath}
\usepackage{graphicx}
\usepackage{float}
\usepackage{epsfig}

\newcommand{\Q}{\mathbb{Q}}
\newcommand{\R}{\mathbb{R}}
\newcommand{\M}{\mathfrak{M}}

\newenvironment{rmks}{\noindent{\bf Remark:}}{}

\newtheorem{thm}{Theorem}[section]
\newtheorem{lem}[thm]{Lemma}
\newtheorem{cor}[thm]{Corollary}
\newtheorem{dfn}[thm]{Definition}

\author{Mart\'\i n Avenda\~no and Ashraf Ibrahim}
\title{Ultrametric Root Counting}
\date{\today}

\begin{document}
 
\maketitle

\begin{abstract}
Let $K$ be a complete non-archimedean field with a discrete valuation,  
$f\in K[X]$ a polynomial with non-vanishing discriminant, 
$A$ the valuation ring of $K$, and $\M$ the maximal ideal of 
$A$. The first main 
result of this paper is a reformulation of Hensel's lemma that connects the 
number of roots of $f$ with the number of roots of its reduction modulo a 
power of $\M$.  We then define a condition --- {\em regularity} --- 
that yields a simple method to compute the exact number of roots of $f$ in $K$.
In particular, we show that regularity implies that the number of roots of $f$ 
equals the sum of the numbers of roots of certain binomials derived from the 
Newton polygon.
\end{abstract}

\section{Introduction}\label{sec-intro} 
Let $f$ be a univariate polynomial with real coefficients. Sturm's 
Theorem \cite{sturm} allows us to determine
the exact number of real roots of $f$ in a given interval $[a,b]$. This is 
done by computing the difference between the number of sign changes of two 
sequences of real numbers called {\em Sturm sequences} \cite{sturm,rs}. We are 
interested in the analogue of Sturm's Theorem over $K$, where $K$ is a 
field, complete with respect to a non-archimdean discrete valuation. More 
precisely, we give an algorithmic method to compute the exact number of roots 
in $K$ (total or with a given valuation) for a large class of polynomials in 
$K[X]$ called {\em regular} polynomials (see definition~\ref{def-reg}).

A classical construction associated to any polynomial $f\in K[X]$ is the 
{\em Newton polygon} (see section \ref{sec-newton} below), which also 
associates monomials of $f$ to points in $\Q^2$. 
For any lower edge $S$ of the Newton polygon of a regular polynomial 
$f\in K[X]$, containing only $2$ points associated to monomials of $f$, the 
binomial containing the corresponding two terms of $f$ is called a 
{\em lower binomial} of $f$. We prove in Theorem~\ref{lower-bin3}
that the number of roots in $K^\ast=K\setminus\{0\}$ of a regular $f$ 
is the sum of the numbers of roots in $K^\ast$ of all its lower binomials.  
A simple explicit formula for the number of roots of each lower binomial 
appears in
Theorem~\ref{lower-bin2}. 

On the other hand, Descartes' rule of signs implies that any univariate 
polynomial $f\in \R[X]$ with exactly $t+1$ monomial terms has at most $2t$ 
non-zero real roots, counted with multiplicity. Note that Descartes' bound 
over the reals doesn't depend on the degree of the polynomial, and is linear 
in the number of monomial terms. In~\cite{Len}, H.\ W.\ Lenstra gave an 
analogue of Descartes' bound over the $p$-adic numbers: if
$f\in K[X]$ has exactly $t+1$ monomial terms and $K$ is a finite extension of 
the $p$-adic rationals $\mathbb{Q}_p$ then the number of roots of $f$ in $K$ 
counted with multiplicity is $O(t^2(q-1)\log t)$ where $q$ is the cardinality 
of the residue field of $K$. As a consequence of our root count 
from Theorem \ref{lower-bin3} we can improve Lenstra's bound to $t(q-1)$ for 
regular polynomials. We 
also prove that our bound for regular polynomials is sharp. For fields of 
non-zero characteristic, our improvement is even greater: B.\ Poonen showed 
in~\cite[Thm. 1]{Poo} that when $p={\rm char}(K)$, the number of roots of a 
sparse polynomial with $t+1$ terms is at most $q^t$, and that there are 
explicit polynomials attaining this bound. Our bound
is linear in $t$ (for regular polynomials) in Poonen's setting as well. 
All this work is done is 
sections~\ref{sec-newton} and~\ref{sec-reg}.

In Theorem \ref{teor1} of section~\ref{sec-mod} we obtain a reformulation of 
the classical construction of Hensel lifting. Let $f\in A[X]$ be a monic 
polynomial with coefficients in the valuation ring $A$ of $K$, and $\bar{f}$ 
the reduction of $f$ modulo $\M^N$ for a sufficiently large integer $N$. We 
give a bijection between the set of roots of $f$ in $K$ and the set of classes 
of roots of $\bar{f}$ in the ring $A/\M^N$ under a particular equivalence 
relation.
As a consequence, for any polynomial in $K[X]$ with non-vanishing discriminant, the number of roots in $K$ depends only on the first
few ``digits'' of the coefficients (see Corollary~\ref{cor-approx}).

\section{Newton Polygon and Regularity}\label{sec-newton}
Let $K$ be a field that is complete with respect to a non-archimedean discrete 
valuation $v$. We denote by
$A=\{x\in K\,:\,v(x)\geq 0\}$ the valuation ring of $K$, $\M=\{x\in K\,:\,v(x)>0\}$ the maximal
ideal of $A$, $\pi\in\M$ a generator of the principal ideal~$\M$ of~$A$, 
and $\kappa=A/\M$ the residue field of~$K$ with respect to~$v$. We assume that 
$\kappa$ is finite with $q$ elements and characteristic $p$ and that 
$v(\pi)=1$. We also denote by $v$ the unique extension of the valuation of 
$K$ to its algebraic closure $\overline{K}$.

\medskip

Let $f(X)=a_nX^n+a_{n-1}X^{n-1}+\dots+a_1X+a_0 \in K[X]$. The 
\emph{Newton polygon} of $f$ is the 
convex hull of the set of points $\{(i,v(a_i))\;:\;i\in\{0,1,\ldots,n\}\}$. 
An edge of a polygon in $\R^2$ is said to be a {\em lower edge} if it 
has an inner normal vector with positive second coordinate. For instance, 
the hexagon that is the convex hull of $\{(-3,1),(-1,0),(1,0),(3,1),(-1,2),
(1,2)\}$ has exactly $3$ lower edges. 

\begin{thm}\label{newton-poly}
Let $f(X)=a_nX^n+a_{n-1}X^{n-1}+\dots+a_1X+a_0 \in K[X]$ be such that 
$n\geq 1$ and $a_0a_n\neq 0$. Let $S$ be a lower edge of the Newton polygon of 
$f$ with vertices $(s,v(a_s))$ and $(s',v(a_{s'}))$ 
with $s>s'$. Then $f$ has exactly $s-s'$ roots in $\overline{K}$, counted with 
multiplicities, with valuation $m$ where $-m$ is the slope of $S$. 
Moreover, $f$ can be factored as
\begin{equation}
f(X)=a_n\prod\limits_{\substack{m=v(\zeta) \\ 
f(\zeta)=0 \ , \  
\zeta \in \overline{K}}} f_m(X)
\label{equ1}
\end{equation}
where, for each $m$, $f_m$ is a non-constant monic polynomial in $K[X]$ with 
all roots of valuation $m$.
\end{thm}
\begin{proof}
See~\cite[Prop. 3.1.1]{Wei}.
\end{proof}

If $S$ is a lower edge of the Newton polygon of $f$ then we will abuse 
notation slightly by also calling $S$ a lower edge of $f$. 
\begin{dfn}\label{def-reg}
A polynomial $f(X)=a_nX^n+a_{n-1}X^{n-1}+\dots+a_1X+a_0 \in K[X]$ is 
\emph{regular} if for any lower edge $S$ of $f$ with vertices 
$(s,v(a_s))$ and $(s',v(a_{s'}))$ with $s>s'$ we have:
\begin{enumerate}
\item $S$ contains exactly two points in the set 
$\{(i,v(a_i))\,:\,i=1,\dots,n\}$.
\item ${\rm char}(\kappa)\nmid (s-s')$. 
\end{enumerate}
The polynomial $a_{s'}X^{s'}+a_sX^s$ is called the \emph{lower binomial} of 
$f$ corresponding to the lower edge $S$.
\end{dfn}

\begin{rmks}  
The notion of regularity introduced in the previous definition is 
not generic in the sense of algebraic geometry, i.e., regularity does not hold 
for all polynomials of degree $n$ with coefficients in a non-empty
Zariski open set in $K^{n+1}$. Nevertheless, regularity has already 
proved quite useful in certain algorithmic questions \cite{airr} and, 
for any choice of exponents, is satisfied by infinitely many polynomials.  
A complete discussion of how likely a given $f\!\in\!K[X]$ is to be regular 
would have to include a discussion of probability measures 
on $\Q_p$ and $\Q_p[X]$, and how they compare with the current 
notions of ``natural'' measures on $\R[X]$. These questions are 
actually far from settled (see, e.g., \cite{edelmankostlan,evans}) and are 
thus beyond the scope of this paper. 
\end{rmks} 

\begin{thm}\label{teor-regularity}
Let $f(X)=X^n+a_{n-1}X^{n-1}+\dots+a_1X+a_0 \in K[X]$ be a regular polynomial. Then all factors $f_m(X)$
in equation (\ref{equ1}) are also regular.
\end{thm}
\begin{proof}
Via Theorem~\ref{newton-poly}, the Newton polygon of the factor $f_m(X)$ 
has exactly $1$ lower edge, lying in the first quadrant and intersecting both 
the coordinate axes, and its slope is $-m\leq 0$ since $f_m$ is monic. 
In particular, all factors $f_m(X)$
satisfy condition (2) of regularity. Therefore it is enough to show that they 
also satisfy condition (1) in definition~\ref{def-reg}.

Let $\alpha_1,\alpha_2,\ldots,\alpha_n\in\overline{K}$ be all the roots of $f$. Assume that

\begin{minipage}{0.45\textwidth}
\begin{align*}
v(a_1) & = \cdots=v(a_{s_1})=m_1 \\
v(a_{s_1+1}) & = \cdots=v(a_{s_2})=m_2 \\
 &  \qquad \vdots  \\
v(a_{s_j+1}) & = \cdots=v(a_{s_{j+1}})=m_{j+1} \\
 &  \qquad \vdots  \\
v(a_{s_{t-1}+1}) & = \cdots=v(a_{n})=m_{t}
\end{align*}
\end{minipage}
\begin{minipage}{0.6\textwidth}
\includegraphics[scale=0.4]{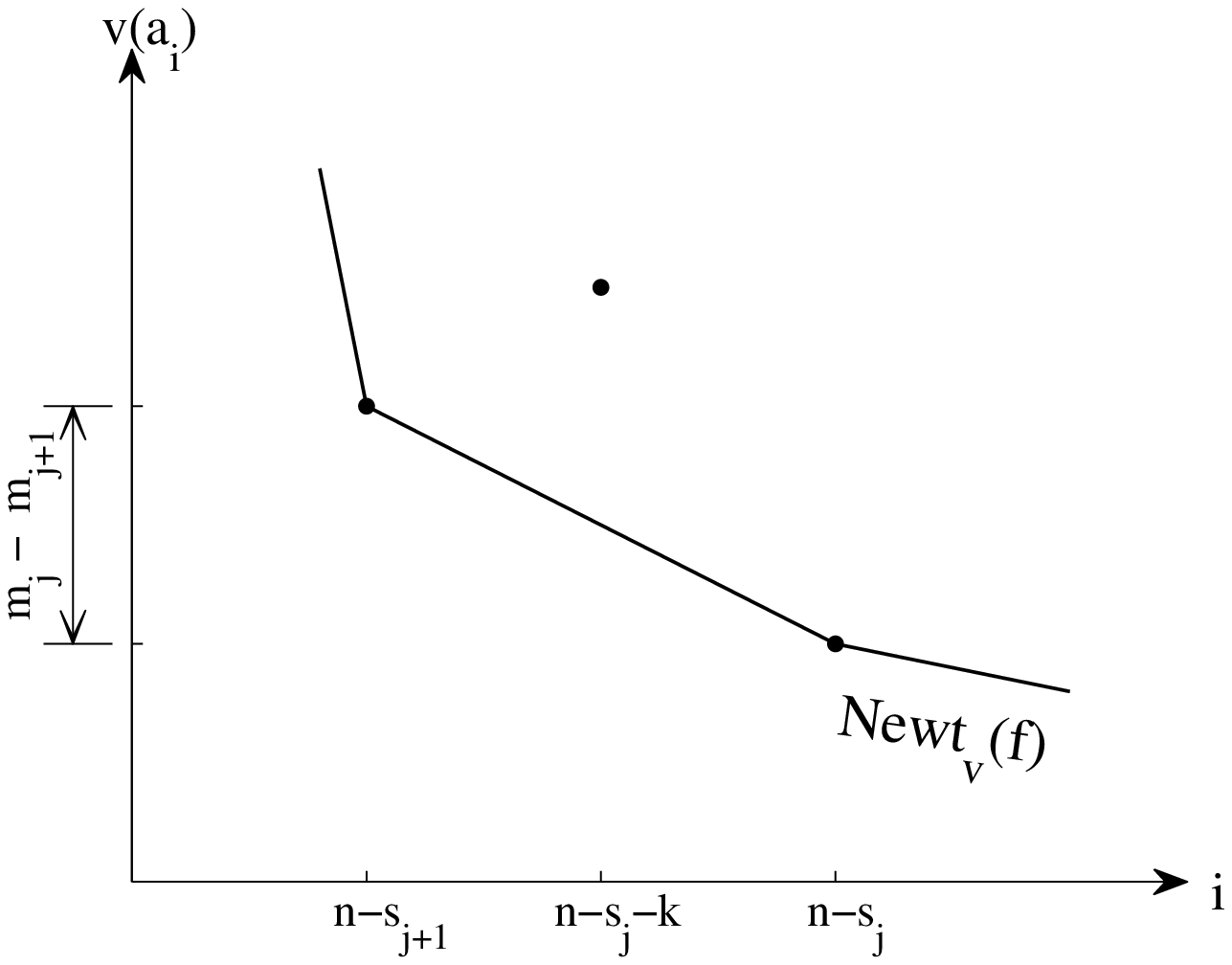}
\end{minipage}

\noindent 
where $m_1<m_2<\dots<m_{t}$. In order to keep consistent notation we set 
$s_0=0$ and $s_{t}=n$. Let $g$ be the factor $f_{m_{j+1}}$ of $f$ and let 
$n_j=s_{j+1}-s_{j}$ be the degree of $g$. Then 
\begin{align*}
 g(X)&=(X-\alpha_{s_{j}+1})(X-\alpha_{s_{j}+2})\cdots(X-\alpha_{s_{j+1}})\\
     &=X^{n_j}+b_{n_j-1}X^{n_j-1}+\cdots+b_1X+b_0.
\end{align*}

The coefficients $b_{n_j-k}$ and $a_{n-s_j-k}$, with $0\leq k\leq n_j$, can be 
written in terms of the roots of $f$ as

\begin{minipage}{0.45\textwidth}
\begin{align*}
b_{n_j-k} & = (-1)^k\!\!\!\sum_{\genfrac{}{}{0pt}{}{I\subseteq \{s_j+1,\ldots,
s_{j+1}\}}{|I|=k}}\!\prod_{i\in I}\alpha_i \\[0.5cm]
a_{n-s_j-k} & = (-1)^{s_j+k}\!\!\!\sum_{\genfrac{}{}{0pt}{}{I\subseteq \{1,
\ldots,n\}}{|I|=s_j+k}}\!\prod_{i\in I}\alpha_i
\end{align*}
\end{minipage}
\begin{minipage}{0.6\textwidth}
\includegraphics[scale=0.4]{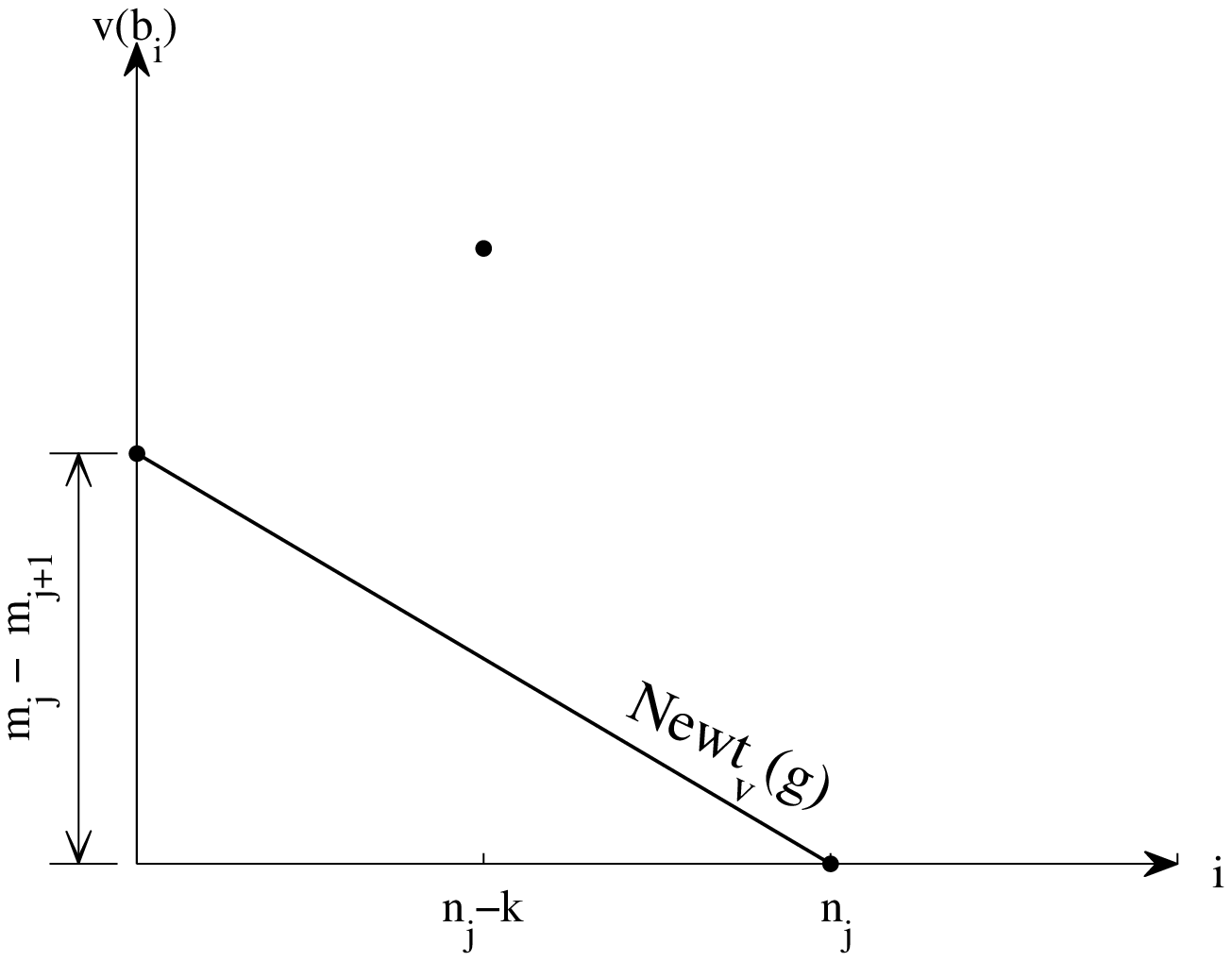}
\end{minipage}

\noindent 
where, as usual, an empty product is defined as $1$.
Note that in the case $k=0$, the term
$\delta=(-1)^{s_j}\alpha_1\alpha_2\cdots\alpha_{s_j}$ appears in the sum 
corresponding to~$a_{n-s_j}$ and has (strictly) the minimum possible valuation. 
This means that $v(\delta)=v(a_{n-s_j})=n_0m_1+n_1m_2+\cdots+n_{j-1}m_j$.
When $0<k<n_j$ we can thus write
\begin{equation}\label{eq1}
a_{n-s_j-k}=\delta b_{n_j-k}+\beta
\end{equation}
where $\beta\in K$ is the sum of all the terms appearing in $a_{n-s_j-k}$ with $I\not\subseteq(0,s_{j+1}]$
or with $I\subseteq (0,s_{j+1}]$ but $I\cap (0,s_j]\neq (0,s_j]$.
This implies that $v(\beta)>n_0m_1+n_1m_2+\cdots+n_{j-1}m_j+km_{j+1}$.
Since $f$ is a regular polynomial, we have that $v(a_{n-s_j-k})>n_0m_1+n_1m_2+\cdots+n_{j-1}m_j+km_{j+1}$ by
the first item in definition~\ref{def-reg}, and hence $v(b_{n_j-k})>km_{j+1}$.
\end{proof}

\section{Roots of the Reduced Polynomial}\label{sec-mod}

Consider a monic polynomial $f(X)=X^n+a_{n-1}X^{n-1}+\dots +a_1X+a_0\in A[X]$.
Assume that the discriminant $\Delta=\textrm{Res}_X(f,f')$ is non-zero and let
$r=v(\Delta)$. In particular we are assuming that $f$ has no multiple roots in $\overline{K}$.
Throughout this section, $f$ and $r$ are fixed. The following lemma is a property of algebraic
integers in the ultrametric setting.

\begin{lem}\label{lem-ria}
For any $\alpha\in\overline{K}$ such that $f(\alpha)=0$, we have $v(\alpha)\geq 0$.
\end{lem}
\begin{proof}
Assume that $v(\alpha)<0$. Since $f(\alpha)=0$, we have that
\begin{align*}
nv(\alpha)&=v(\alpha^n)=v(a_{n-1}\alpha^{n-1}+\cdots+a_0)\geq\min\{v(a_i\alpha^i)\,:\,0\leq i<n\} \\
          &\geq\min\{v(\alpha^i)\,:\,0\leq i<n\}=(n-1)v(\alpha)
\end{align*}
which implies $v(\alpha)\geq 0$, a contradiction.
\end{proof}

The following lemma gives a lower bound for the distance between roots in terms
of the valuation $r$ of the discriminant.
\begin{lem}\label{lem-dist}
If $f(X)=\prod_{i=1}^{n}(X-\alpha_i)$ with $\alpha_i\in\overline{K}$ for $i=1,\ldots,n$, then
$$v(\alpha_i-\alpha_j)\leq\frac{r}{2}\qquad \forall\, i\neq j.$$
\end{lem}
\begin{proof}
From the formula of the discriminant
$\Delta=\prod_{1\leq i<j\leq n}(\alpha_i-\alpha_j)^2$
we get $r=2\sum_{1\leq i<j\leq n}v(\alpha_i-\alpha_j)$. Since all the roots satisfy $v(\alpha_i)\geq 0$,
all the terms in this sum are non-negative. Therefore $v(\alpha_i-\alpha_j)$ can not exceed $r/2$ for any
$i\neq j$.
\end{proof}

The bound of Lemma~\ref{lem-dist} is sharp. For instance, the polynomial $f=x(x-p)\in\mathbb{Q}_p[X]$ has discriminant $\Delta(f)=p^2$
of valuation $v_p(\Delta(f))=2$, and the valuation of the difference of the roots is $1$.
We can also use Lemma~\ref{lem-dist} to derive an upper bound for the number of roots of $f$ in $K$.

\begin{cor}\label{cor-ub}
The number of roots of $f$ in $K$ is not greater than $q^{[r/2]+1}$.
\end{cor}
\begin{proof}
Otherwise we would have two roots $x,y\in A$ with $x\equiv y\bmod{\pi^{[r/2]+1}}$, that is,
$v(x-y)\geq [r/2]+1>r/2$ in contradiction with Lemma~\ref{lem-dist}.
\end{proof}

Let $f_N\in \left(A/\pi^NA\right)[X]$ denote the reduction of the polynomial $f$ modulo $\pi^N$.  We denote by
$\beta_1,\ldots,\beta_l\in A$ the roots of $f$ in $K$ (by Lemma~\ref{lem-ria} we know that they are in $A$).
It is clear that the reduction of any of these roots modulo $\pi^N$ is a root of $f_N$. Unfortunately, the reduction
modulo $\pi^N$ does not give a bijection between the set of roots of $f$ in $K$ and the set of roots of~$f_N$ in
$A/\pi^NA$ in general. However, we will show that the reduction homomorphism 
is a bijection between the roots of $f$ and classes of roots of $f_N$ under 
a particular equivalence relation. The inverse of the reduction homomorphism
is given by a reformulation of the standard Hensel's lemma.

\bigskip

We denote by $\overline{x}$ the reduction modulo $\pi^NA$ of $x\in A$.

\begin{dfn}
Let $S_N\subseteq A/\pi^NA$ be the set of roots of $f_N$. Two roots $x,y\in S_N$ are in the same equivalence class
(denoted by $x\approx y$) if and only if either $x=y$ and $N\leq r$ or $x\equiv y\bmod{\overline{\pi}^{r+1}}$ and $N>r$.
The class containing a root $x\in S_N$ is written $[x]$ and the set of classes is written $S_N/\approx$.
\end{dfn}

\begin{lem}\label{lem1}
If $N>r$ then the number of roots of $f$ in $K$ is not greater than $|S_N/\approx|$.
\end{lem}
\begin{proof}
Write $f(X)=(X-\beta_1)(X-\beta_2)\dots(X-\beta_l)g(X)$ where $g$ has no roots 
in $K$.
Let $\beta_{i,N}=\overline{\beta_i}\in A/\pi^NA$ be the reduction of $\beta_i$ 
modulo $\pi^NA$. Since
this reduction is a ring homomorphism, $\beta_{i,N}$ is a root of $f_N$. Take 
$1\leq i<j\leq l$. By
Lemma~\ref{lem-dist}, we have $v(\beta_i-\beta_j)\leq r/2\leq r$, i.e., 
$\beta_i\not\equiv\beta_j\bmod{\pi^{r+1}}$.
Since $N>r$, we also have that $\overline{\beta_i}\not\equiv\overline{\beta_j}\bmod{\overline{\pi}^{r+1}}$.
This implies that $\beta_{i,N}\not\approx\beta_{j,N}$ and hence $[\beta_{i,N}]\neq[\beta_{j,N}]$.
\end{proof}

\begin{lem}\label{lem-res}
Let $\gamma\in A$ be such that $v(f(\gamma))>r$. Then $v(f'(\gamma))\leq r$.
\end{lem}
\begin{proof}
Write $\Delta=a(X)f(X)+b(X)f'(X)$ with $a,b\in A[X]$ and evaluate at $X=\gamma$. Since $v(a(\gamma))\geq 0$, we have that
$v(a(\gamma)f(\gamma))>r$, and therefore $v(b(\gamma)f'(\gamma))=v(\Delta-a(\gamma)f(\gamma))=r$. We conclude that $v(f'(\gamma))\leq r$
because $v(b(\gamma))\geq 0$.
\end{proof}

In order to proceed, we need the following version of Hensel's lemma. This lemma allows us to lift an approximate
root of $f$ to an exact root.

\begin{lem}[Hensel]
If $\gamma\in A$ satisfies $v(f(\gamma)/f'(\gamma)^2)>0$ then there exists a root $\xi\in A$
of $f$ such that $v(\xi-\gamma)=v(f(\gamma)/f'(\gamma))$.
\end{lem}
\begin{proof}
See~\cite[Sec. 1.5, Ch. 2]{Rob}.
\end{proof}

\begin{lem}\label{lem2}
If $N>2r$ then the number of roots of $f$ in $K$ is not less than $|S_N/\approx|$.
\end{lem}
\begin{proof}
Take $[\beta]\in S_N/\approx$ and take some $\gamma\in A$ such that $\beta=\overline{\gamma}$.
Since $\overline{f(\gamma)}=f_N(\beta)=\overline{0}$, we have that $v(f(\gamma))\geq N>2r\geq r$.
By Lemma~\ref{lem-res} we have that $v(f'(\gamma))\leq r$ and then $v(f(\gamma)/f'(\gamma)^2)>0$.
By Hensel's lemma, there exists $\xi\in A$ such that $f(\xi)=0$ and $\xi\equiv\gamma\bmod{\pi^{N-r}}$
because $v(f(\gamma)/f'(\gamma))\geq N-r$.
Since $N-r>r$ we have that $\xi\equiv\gamma\bmod{\pi^{r+1}}$ and also
$\overline{\xi}\equiv\overline{\gamma}\bmod{\overline{\pi}^{r+1}}$ because $N>r$.
This means that $[\beta]=[\overline{\xi}]$.

Note that if $\xi$ and $\xi'$ are two different roots of $f$ in $A$, then $v(\xi-\xi')\leq r/2\leq r$ by Lemma~\ref{lem-dist}.
This implies that $\xi\not\equiv\xi'\bmod{\pi^{r+1}}$, $\overline{\xi}\not\equiv{\overline{\xi'}}\bmod{\overline{\pi}^{r+1}}$ and
$[\overline{\xi}]\neq[\overline{\xi'}]$. We conclude from here that the procedure described above gives a well defined map
from the set $S_N/\approx$ to the set of roots of $f$ in $K$ (we can not lift the same class to two different roots).
Moreover, this map is injective, because it is possible to reconstruct the equivalence class from the lifted root.
\end{proof}

As an immediate consequence of Lemmas~\ref{lem1} and~\ref{lem2}, we obtain a bijection between the number
of roots of $f$ in $K$ and the number of equivalence classes. The following theorem is the main result of this
section.

\begin{thm}\label{teor1}
For any $N>2r$, the number of roots of $f$ in $K$ is equal to $|S_N/\approx|$. More precisely, the
map $x\mapsto[\overline{x}]$ is a bijection between the set of roots of $f$ in $A$ (or in $K$) and
$S_N/\approx$.
\end{thm}

\begin{cor}\label{cor-approx}
Let $g=X^n+b_{n-1}X^{n-1}+\cdots+b_0\in A[X]$ be a polynomial such that $v(a_i-b_i)>2r$. Then $f$ and $g$ have the same
number of roots in $K$.
\end{cor}
\begin{proof}
Since $a_i\equiv b_i\bmod{\pi^{2r+1}}$, then $${\rm Res}_X(g,g')\equiv{\rm Res}_X(f,f')\equiv\Delta\bmod{\pi^{2r+1}}.$$
Therefore the discriminant of $g$ has also valuation $r$. We conclude by applying Theorem~\ref{teor1} to $f$ and $g$
with $N=2r+1$.
\end{proof}

It is important to note that the proofs of Lemmas~\ref{lem1} and~\ref{lem2} remain valid if we change our
equivalence relation $\approx$ by the (apparently finer) relation $\sim$ defined by $x\sim y$ if and
only if $x\equiv y\bmod{\overline{\pi}^{N-r}}$. Therefore Theorem~\ref{teor1} remains true with this new
equivalence relation. Denote by $[[x]]$ the equivalence class of roots with respect to $\sim$ that contains
$x$. It is clear that $[[x]]\subseteq[x]$ for all $x$. On the other hand, the number of classes with respect to
$\sim$ or $\approx$ must be the same (they coincide with the number of roots of $f$ in $K$), thus $[[x]]=[x]$ for
all roots $x\in A/\pi^NA$ of $f_N$. We derive several corollaries from this remark.

\begin{cor}\label{cor-ub2}
For any $N>2r$, the number of roots of $f_N$ in $A/\pi^NA$ is less than or equal to $q^r$ times the number of
roots of $f$ in $K$.
\end{cor}
\begin{proof}
Any class $[[x]]$ contains at most $q^r$ elements and the number of classes is the number of roots of $f$ in $K$.
\end{proof}

\begin{cor}
For any $N>2r$, the number of roots of $f_N$ in $A/\pi^NA$ is not greater than $q^{r+[r/2]+1}$.
\end{cor}
\begin{proof}
Apply Corollaries~\ref{cor-ub} and~\ref{cor-ub2}.
\end{proof}

\begin{cor}
If $r=0$ then the number of roots of $f_N$ in $A/\pi^NA$ coincide with the number of roots of $f$ in $K$ for all $N\geq 1$.
\end{cor}
\begin{proof}
Apply Corollary~\ref{cor-ub2} and Lemma~\ref{lem1}.
\end{proof}

\section{Roots of Regular Polynomials}\label{sec-reg}

The goal of this section is to give a procedure to count the exact the number of roots in $K^\ast$ of regular polynomials.
This is done in Theorems~\ref{lower-bin2} and \ref{lower-bin3}.
The following corollary is just a special case of Theorem~\ref{teor1}, when $r=0$ and $N=1$, but we are going to use
it in this section, so we would like to state it as a separate result. It should also be pointed out the both Corollary~\ref{cor-approx2}
and Lemma~\ref{lem-bin} are standard results.

\begin{cor}\label{cor-approx2}
If $v(\Delta)=0$ then the number of roots of $f$ in $K^\ast$ is equal to the number of roots of $f_1$ in $\kappa^\ast$ where $f_1$ is the
reduction of $f$ modulo $\pi A$. 
\end{cor} 

\begin{lem}\label{lem-bin}
If $f(X)=X^n+a_0$ then the discriminant of $f$ is
$$\Delta(f)=(-1)^{n(n-1)/2}n^na_0^{n-1}.$$
\end{lem}
\begin{proof}
Write $f(X)=X^n+a_0=\prod_{i=1}^{n}(X-\alpha_i)$ with $\alpha_i\in\overline{K}$. Then
\begin{align*}
\Delta(f)&= (-1)^{n(n-1)/2}\textrm{Res}(f,f')\!=\!(-1)^{n(n-1)/2}\prod_{i=1}^nf'(\alpha_i)\!=\!(-1)^{n(n-1)/2}\prod_{i=1}^nn\alpha_i^{n-1}\\
&=(-1)^{n(n-1)/2}n^n\left(\prod_{i=1}^n\alpha_i\right)^{n-1}=(-1)^{n(n-1)/2}n^n(-1)^{n(n-1)}a_0^{n-1}\\
&=(-1)^{n(n-1)/2}n^na_0^{n-1}.
\end{align*}
\end{proof}

\begin{lem}\label{lower-bin1}
If $g(X)=X^n+a_{n-1}X^{n-1}+\dots+a_1X+a_0\in A[X]$ satisfies $v(a_0)=0$, $v(a_i)>0$ for all $1\leq i<n$ and
$p\nmid n$ then the number of roots of $g$ in $K^\ast$ is equal to the number of roots of the lower binomial
$X^n+a_0$ of $g$ in $K^\ast$.
\end{lem}
\begin{proof}
By Lemma~\ref{lem-bin}, the discriminant of $X^n+a_0$ has valuation $0$. On the other hand, the polynomial $g$
satisfies the hypothesis of Corollary~\ref{cor-approx} with respect to $f=X^n+a_0$. Then both $g$ and its lower
binomial $f$ have the same number of roots in $K$.
\end{proof}

\begin{dfn}
Let $a\in K^\ast$ be an element with valuation $v(a)=l$. The first non-zero digit of $a$ is $\delta(a)=\overline{a/\pi^l}\in \kappa^\ast$.
\end{dfn}

The following result gives a procedure to count the number of roots of a 
regular polynomial when its Newton polygon has only one lower edge.

\begin{thm}\label{lower-bin2}
Let $f(X)=X^n+a_{n-1}X^{n-1}+\dots+a_1X+a_0\in A[X]$ with $p\nmid n$ and 
$a_0\neq 0$. Write $l=v(a_0)$ and assume that
$v(a_{n-i})>il/n$ for all $i=1,\ldots,n-1$. Then the number $R$ of roots of 
$f$ in $K^\ast$ is equal to the number of roots
of the lower binomial $X^n+a_0$ in $K^\ast$. Moreover, if $n\nmid l$ we have 
$R=0$, and if $n|l$ then 
\begin{equation*}
R = \left\{
 \begin{array}{cl} 
  gcd(n,q-1) & \text{if } -\delta(a_0) \textrm{ is an $n^{th}$ power 
   in }\kappa, \\
   0 & \text{otherwise.}
 \end{array} \right.
\end{equation*}
\end{thm}
\begin{proof}
By Theorem~\ref{newton-poly}, all the roots of both $f$ and $\tilde{f}=X^n+a_0$ have valuation $e=l/n$. It is clear
that if $n\nmid l$, then neither $f$ nor $\tilde{f}$ have a root in $K$, because all the elements in $K$ have
integer valuation. Therefore, we only need to consider the case $n|l$.

Define $h(X)=\pi^{-l}f(\pi^eX)$. It is clear that $f$ and $h$ have the same number of roots in~$K$. Our assumptions
on the coefficients of $f$ guarantee that $h$ is a monic polynomial in $A[X]$. Moreover, if $h=X^n+b_{n-1}X^{n-1}+\cdots+b_0$,
then $v(b_0)=0$ and $v(b_{n-i})>0$ for all $1\leq i<n$. By Lemma~\ref{lower-bin1}, the number of roots of $h$ in~$K$
coincides with the number of roots of its lower binomial $\tilde{h}=X^n+\pi^{-l}a_0$ in $K$. Since
$\tilde{h}(X)=\pi^{-l}\tilde{f}(\pi^{e}X)$, then $\tilde{f}$ and $\tilde{h}$ have the same number of roots in $K$.
We conclude that $f$, $\tilde{f}$, $h$ and $\tilde{h}$ have all the same number $R$ of roots in $K$.

It only remains to prove the formula for $R$. By Lemma~\ref{lem-bin}, the discriminant of $\tilde{h}$ has valuation $0$
(since $p\nmid n$ and $v(b_0)=0$). Therefore, by Corollary~\ref{cor-approx2}, the number of roots $R$ of $\tilde{h}$ in $K$
equals the number of roots in $\kappa$ of the reduction $\tilde{h}_1=X^n+\delta(a_0)$ of $\tilde{h}$ modulo $\M$.
If $-\delta(a_0)$ is not an $n^{th}$ power in $\kappa$, then $\tilde{h}$ has no roots. Otherwise, the number of roots of $\tilde{h}$
in~$\kappa$ coincides with the number of $n^{th}$ roots of the unity in~$\kappa$. Since $\kappa^\ast$ is a cyclic group with $q-1$
elements, $R=\gcd(q-1,n)$ in this case.
\end{proof}

\begin{thm}\label{lower-bin3}
Let $f=a_nX^n+\cdots+a_0\in K[X]$ be a regular polynomial. Then the number of roots of $f$ in $K^\ast$ is equal to
the sum of the number of roots in $K^\ast$ of all its lower binomials.
\end{thm}
\begin{proof}
By Theorem~\ref{newton-poly}, we can write $f=a_n\prod_{j=0}^{t}f_j$ where
$f_0,\ldots,f_t\in K[X]$ are monic polynomials and all the roots of each $f_j$ 
have the same valuation $m_{j+1}$. Here $t+1$ is the number of lower edges of 
the Newton polygon of $f$ and $-m_1>\cdots>-m_{t+1}$ are the slopes of the 
lower edges. Following the notation of
Theorem~\ref{teor-regularity} we define $n_{j+1}=\deg(f_j)$ and 
$s_j=|\{\alpha\in\overline{K}\;:\;f(\alpha)=0\;\mbox{and}\;
v(\alpha)\leq m_j\}|$.
Setting $s_0=0$ we have $n_j=s_{j+1}-s_j$.
The lower binomials of $f$ are the polynomials 
$g_j=a_{n-s_j}X^{n-s_j}+a_{n-s_{j+1}}X^{n-s_{j+1}}$.
Let $R$ and $R_j$ denote the number of roots in $K^\ast$ of $f$ and $f_j$ 
respectively. It is
clear that $R=R_0+\cdots+R_t$. By Theorem~\ref{teor-regularity} the polynomials 
$f_j$ are regular, and then, by Theorem~\ref{lower-bin2} its number $R_j$ of 
roots in $K^\ast$ depends only on its degree and the first digit of its 
constant term. In order to
conclude we only need to prove that $R_j$ coincides with the number of roots of $g_j$ in $K^\ast$.
The number of roots of the lower binomial $g_j=a_{n-s_j}X^{n-s_{j+1}}(X^{s_{j+1}-s_j}+a_{n-s_{j+1}}/a_{n-s_j})$
in $K^\ast$ coincide with the number of roots of the regular monic polynomial 
$X^{s_{j+1}-s_j}+a_{n-s_{j+1}}/a_{n-s_j}$ in $K^\ast$. The degree of this polynomial is $n_j=\deg(f_j)$
and by the equation~\ref{eq1} (with $k=n_j$) in the proof of Theorem~\ref{teor-regularity}, the first digit of $a_{n-s_{j+1}}/a_{n-s_j}$ is equal to the first digit of the constant term of $f_j$. Therefore $R_j$
is also the number of roots of $g_j$ in $K^\ast$.
\end{proof}

\begin{cor}\label{cor-lb}
Let $f\in K[X]$ be a regular polynomial with $t+1$ terms. Then the number of 
roots of $f$ in $K^\ast$ is at most $t(q-1)$, and all the roots of 
$f$ in $K^\ast$ are simple.
\end{cor}
\begin{proof}
The number of lower binomials (i.e., the number of lower edges of the Newton 
polygon) of $f$ is bounded above by $t$. By Theorem~\ref{lower-bin2}, the 
number of non-zero roots of each lower binomial is at most $q-1$. Using 
Theorem~\ref{lower-bin3} we conclude that $f$ has at most $t(q-1)$ roots in 
$K^\ast$.
\end{proof}

We conclude this section by showing that the bound of Corollary~\ref{cor-lb} 
is sharp. Consider the polynomial
$$f=\sum_{i=0}^t(-1)^i\pi^{i^2(q-1)}X^{i(q-1)}\in K[X].$$
It is then easily verified that the Newton polygon of $f$ has 
exactly $t$ lower edges, and their vertices consists of pairs of the form  
$$\left\{\left(i(q-1),i^2(q-1)\right),\left((i+1)(q-1),(i+1)^2(q-1)\right)
\right\}$$
for all $i\in\{0,\ldots,t-1\}$. The polynomial $f$ is regular: $f$ satisfies 
the first item of definition~\ref{def-reg} because all the coefficients of $f$
correspond to vertices of the lower hull of the Newton polygon, and $f$ 
satisfies the second item since $p={\rm char}(\kappa)$ is coprime to
$(i+1)(q-1)-i(q-1)=q-1$. The lower binomials of $f$ are
$$f_i=(-1)^{i+1}\pi^{(i+1)^2(q-1)}X^{(i+1)(q-1)}
+(-1)^i\pi^{i^2(q-1)}X^{i(q-1)}\in K[X]$$
for all $i\in\{0,\ldots,t-1\}$. The number of roots of $f_i$ in $K^\ast$ 
coincides with the number of roots of $X^{q-1}-\pi^{(2i+1)(q-1)}$, which is
$q-1$ according to Theorem \ref{lower-bin2}. Moreover, by Theorem 
\ref{lower-bin3}, the number of roots of $f$ in $K^\ast$ is the sum of the
number of roots of the lower binomials $f_i$ in $K^\ast$. This proves that $f$ 
has exactly $t(q-1)$ roots in $K^\ast$.

\section{Conclusion}\label{sec-concl}

Our root counting method, given in Theorems \ref{lower-bin2} and 
\ref{lower-bin3}, works only with regular polynomials.
Is it possible to give a similar procedure for general polynomials? We believe 
that the result in Theorem \ref{teor1} could be a first step in that 
direction. We also ask whether the upper bound in Corollary~\ref{cor-lb} can  
be extended to a larger class of polynomials. Finally, we point out that  
an extension of regularity to the multivariate case was initiated in 
\cite{Ibrahim}.

\section*{Acknowledgments}

The authors would like to thank J.\ Maurice Rojas for many fruitful discussions.


\begin{thebibliography}{99} 
\bibitem{airr} M.\ Avenda\~{n}o, A.\ Ibrahim, J.\ M.\ Rojas, 
and K.\ Rusek, {\it ``Randomized NP-Completeness for
$p$-adic Rational Roots of Sparse Polynomials in One Variable,''} 
proceedings of ISSAC 2010, ACM Press, to appear.  
\bibitem{edelmankostlan} A.\ Edelman and E.\ Kostlan, 
{\it ``How Many Zeros of a Random Polynomial are Real?,"} Bull.\ Amer.\
Math.\ Soc., 32, January (1995), pp.\ 1--37. 
\bibitem{evans} Steven N.\ Evans, {\it ``The expected number of zeros of a 
random system of $p$-adic polynomials,''} Electron.\ Comm.\ Probab.\ 11 
(2006), pp.\ 278-290.
\bibitem{Ibrahim} A.\ Ibrahim, \emph{Ultrametric Fewnomial Theory,} 
Mathematics Ph.D.\ thesis, Texas A\&{}M University (Dec.\ 2009), University 
Microfilms, Ann Arbor, Michigan.  
\bibitem{Len} H.\ W.\ Lenstra, \emph{On the Factorization of Lacunary 
Polynomials}, Number Theory in Progress, vol.\ 1, Berlin 1999, pp.\ 277--291.
\bibitem{Poo} B.\ Poonen, \emph{Zeros of sparse polynomials over local fields 
of characteristic $p$}, Math.\ Res.\ Lett.\ 5(3), pp.\ 273--279, 1998.
\bibitem{rs} Q.\ I.\ Rahman and G.\ Schmeisser, 
{\it Analytic Theory of Polynomials,} Clarendon Press,
London Mathematical Society Monographs 26, 2002.
\bibitem{Rob} A. Robert, \emph{A course in p-adic Analysis}, GTM, Vol.198, 
Springer-verlag, 2000.
%\bibitem{Sch} W.H. Schikhof, \emph{Ultrametric Calculus, An Introduction to 
% $p$-adic Analysis}, Cambridge Studies in
% Adv.Math. 4, Cambridge Univ. Press, 1984.
\bibitem{sturm} C.\ Sturm, {\it
``M\'emoire sur la r\'esolution des \'equations num\'eriques,''}
Inst.\ France Sc.\ Math.\ Phys., {\bf 6} (1835).
\bibitem{Wei} E. Weiss, \emph{Algebraic Number Theory}, McGraw-Hill, New York 
1963.
\end{thebibliography}
\end{document}